\documentclass[onefignum,onetabnum]{siamart190516}

\usepackage{amsfonts}
\usepackage{graphicx}
\usepackage{epstopdf}
\usepackage{algorithm}
\usepackage{algpseudocode}

\algdef{SE}[SUBALG]{Indent}{EndIndent}{}{\algorithmicend\ }%
\algtext*{Indent}
\algtext*{EndIndent}
\graphicspath{{./Images/}}

\ifpdf
  \DeclareGraphicsExtensions{.eps,.pdf,.png,.jpg}
\else
  \DeclareGraphicsExtensions{.eps}
\fi

\newsiamremark{remark}{Remark}
\newsiamremark{hypothesis}{Hypothesis}
\crefname{hypothesis}{Hypothesis}{Hypotheses}
\newsiamthm{claim}{Claim}

\headers{Rank-Restricted Soft SVD Convergence}{M. Panagoda, T. Berry, and H. Antil}

\title{Convergence Analysis of the Rank-Restricted Soft SVD Algorithm \thanks{Submitted to the editors DATE.
\funding{This work is partially supported by NSF grants DMS-1854204, DMS-2006808, DMS-1818772, DMS-1913004, the Air Force Office of Scientific Research (AFOSR) under Award NO: FA9550-19-1-0036, and Department of Navy, Naval PostGraduate School under Award NO: N00244-20-1-0005.}}}

\author{Mahendra Panagoda\thanks{Department of Mathematical Sciences, George Mason University, Fairfax, VA.}
\and Tyrus Berry\footnotemark[2] \thanks{\email{tberry@gmu.edu}}
\and Harbir Antil\footnotemark[2]}

\usepackage{amsopn}

\ifpdf
\hypersetup{
  pdftitle={Convergence Analysis of the Rank-Restricted Soft SVD Algorithm},
  pdfauthor={M. Panagoda, T. Berry, and H. Antil}
}
\fi

\begin{document}

\maketitle

\begin{abstract}
The soft SVD is a robust matrix decomposition algorithm and a key component of matrix completion methods.  However, computing the soft SVD for large sparse matrices is often impractical using conventional numerical methods for the SVD due to large memory requirements.  The Rank-Restricted Soft SVD (RRSS) algorithm introduced by Hastie et al. addressed this issue by sequentially computing low-rank SVDs that easily fit in memory.  We analyze the convergence of the standard RRSS algorithm and we give examples where the standard algorithm does not converge.  We show that convergence requires a modification of the standard algorithm, and is related to non-uniqueness of the SVD.   Our modification specifies a consistent choice of sign for the left singular vectors of the low-rank SVDs in the iteration.  Under these conditions, we prove linear convergence of the singular vectors using a technique motivated by alternating subspace iteration.  We then derive a fixed point iteration for the evolution of the singular values and show linear convergence to the soft thresholded singular values of the original matrix.  This last step requires a perturbation result for fixed point iterations which may be of independent interest. 
\end{abstract}

\begin{keywords}
  low rank approximation, soft SVD, matrix completion, regularization
\end{keywords}

\begin{AMS}
  65F55, 65F22, 15A83
\end{AMS}

\section{The Rank-Restricted Soft SVD}


In this paper we consider the following rank-restricted matrix decomposition problem,
\begin{equation}\label{rrss} \underset { A \in \mathbb{R}^{n \times r},  B \in \mathbb{R}^{m \times r} }{ \min } \quad  \frac{1}{2}\|  X - AB^\top \|^2_F  +  \frac{\lambda}{2}(\|A\|^2_F+ \|B\|^2_F) \end{equation}
where $X \in \mathbb{R}^{n\times m}$ is considered the input to the problem, $r \leq p \equiv \min\{m,n\}$ is the rank restriction, and $\lambda$ is a regularization parameter.  The product $AB^\top$ is an approximation of $X$ in the Frobenius norm with rank at most $r$.  In \cite{mazumder2010spectral,Hastie} it was shown that when $A,B$ solve \eqref{rrss} the product $AB^\top$ solves,
\begin{equation}\label{ssvd} \underset {Z \, : \, \textup{rank}(Z)\leq r}{ \min } \quad  \frac{1}{2}\|  X - Z \|^2_F  +  \lambda\|Z\|_* \end{equation} 
where the nuclear norm $\|Z\|_*$ is the sum of the singular values of $Z$.  The relationship between these solutions suggests that $AB^\top$ is a robust low-rank approximation to $X$.  This approximation is a key component of many matrix completion algorithms \cite{Hastie,mazumder2010spectral,cai2010singular,candes2010matrix}.  In this paper we will analyze a numerical method for solving \eqref{rrss} proposed by Hastie et al. in \cite{Hastie}.  We will show that a modification is required to obtain convergence, and we give the first complete proof of convergence.

The problem \eqref{rrss} is called the Rank-Restricted Soft SVD (RRSS) because the solution involves soft-thresholding of the singular value decomposition (SVD). Given the reduced SVD, $X=USV^\top$ ($U\in \mathbb{R}^{n\times p}$, $S\in \mathbb{R}^{p\times p}$, $V\in \mathbb{R}^{m\times p}$ where $p\equiv \min\{m,n\}$), the solution to \eqref{rrss} is found by first soft-thesholding the singular values,
\[ D \equiv \sqrt{(S-\lambda I)^+ }= \sqrt{\max\{0,S-\lambda I \}} \]
 and then defining $A_{\rm opt}=UDI_{p\times r}$ and $B_{\rm opt} = VD^\top I_{p\times r}$ \cite{Hastie}.  When $X$ is full matrix, a standard or partial SVD can be used to obtain this solution.  However, in many applications such as matrix completion, $X$ is a sparse matrix that is too large to be stored as a full matrix.  Motivated by these applications, in \cite{Hastie} Hastie et al. introduced a fast and memory efficient alternating ridge regression algorithm shown as Algorithm \ref{alg1} below.

\begin{algorithm}[H]\label{alg1} 
\centering
\caption{Alternating Directions Optimization for \eqref{rrss}} 
\begin{algorithmic} 
	\State {\bf Inputs:} An $n\times m$ matrix $X$, rank restriction $r$, and regularization parameter $\lambda$
	\State {\bf Outputs:} An $n\times r$ matrix $A$ and an $m\times r$ matrix $B$
	\State
    \State Initialize $A$ as a random $n\times r$ matrix  and $A_{\rm p} = B_{\rm p} = 0$
    \While{$\frac{||A-A_{\rm p}||_{\max}}{||A||_{\max}} + \frac{||B-B_{\rm p}||_{\max}}{||B||_{\max}} > {\rm tol}$}
            \State $A_{\rm p} = A$, $B_{\rm p} = B$
            \State Update $B$ leaving $A$ fixed:
            \Indent
            \State $B \leftarrow X^\top A(A^\top A + \lambda I_{r\times r})^{-1}$
            \EndIndent
            \State Update $A$ leaving $B$ fixed:
            \Indent
            \State $A \leftarrow X B(B^\top B + \lambda I_{r\times r})^{-1}$
            \EndIndent
    \EndWhile
\end{algorithmic}
\end{algorithm}

We first consider a simplistic approach to solving \eqref{rrss} shown in Algorithm \ref{alg1}.  This method is motivated by the alternating directions method of optimization \cite{parikh2014proximal,boyd2011distributed}. The objective function in \eqref{rrss} is not convex as a function of both $A$ and $B$ together, however, when either $A$ or $B$ is fixed the objective function is convex and quadratic in the other.  For example when $A$ is fixed, we can rewrite the objective function in \eqref{rrss} as,
\[ \sum_{i=1}^N \frac{1}{2}||X_i-AB_i ||_2^2 + \frac{\lambda}{2}||B_i||_2^2 + c_1 = \sum_{i=1}^N \frac{1}{2} B_i^\top (A^\top A+\lambda I_{r\times r}) B_i -  B_i^\top A^\top X_i + c_2 \]
where $X_i$ is the $i$-th column of $X$ and $B_i$ is the $i$-th column of $B^\top$ ($c_1,c_2$ are constants with respect to $B$).  The optimization problems for each column of $B^\top$ are independent and the optimal solution is $B_i = (A^\top A+\lambda I_{r\times r})^{-1}A^\top X_i$.  Combining these columns we find the optimal solution for $B$, when $A$ is fixed, has the closed form solution, $X^\top A(A^\top A + \lambda I_{r\times r})^{-1}$.  If we then hold $B$ fixed, we have a similar optimization problem for $A$ with optimal solution $XB(B^\top B +\lambda I_{r\times r})^{-1}$.

\begin{figure}[h]
\includegraphics[width=0.45\linewidth]{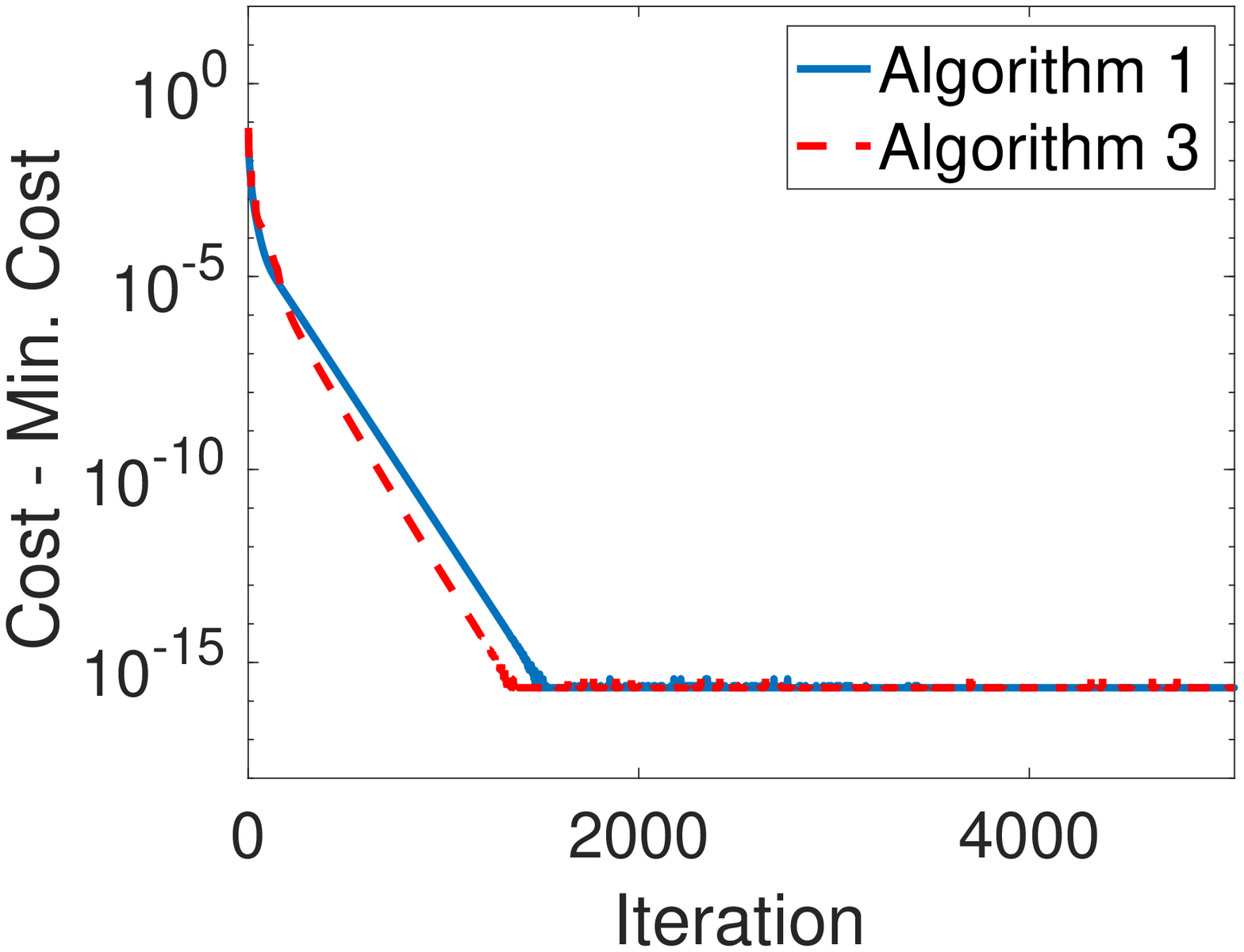}\hspace{10pt}\includegraphics[width=0.45\linewidth]{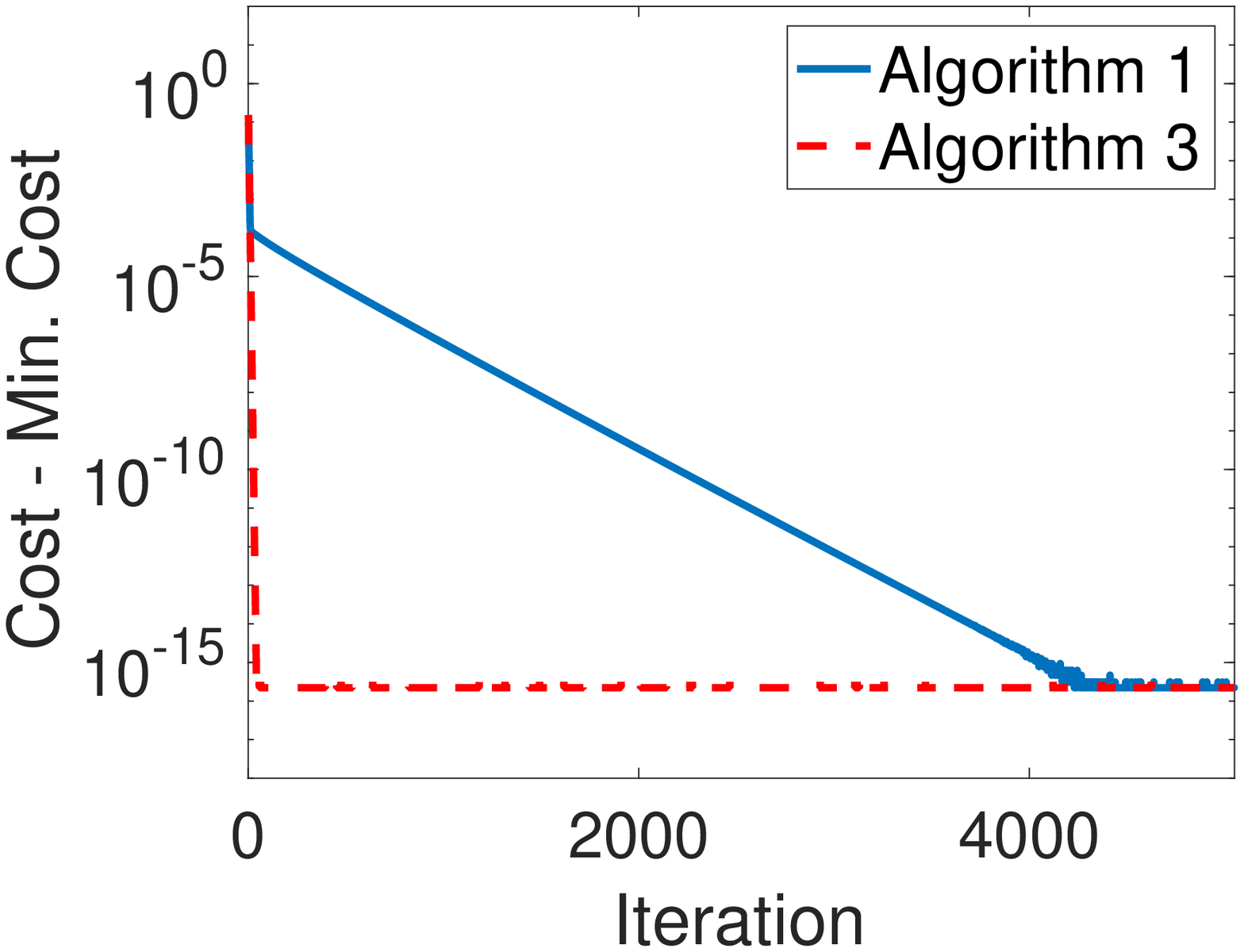}
\caption{\label{fig1} For a full-rank $X$ matrix (left) Algorithms \ref{alg1} and \ref{alg3} have similar performance, however when $X$ is approximately low-rank (right) Algorithm \ref{alg3} is significantly faster.  In both examples $X$ is $500\times 500$ and we set $r=10$ and $\lambda=0.5$ and the minimum cost is computed using $A_{\rm opt}, B_{\rm opt}$.  In the left panel the entries of $X$ are independent standard Gaussian random variables.  In the right panel $\tilde A,\tilde B$ are $500\times 10$ matrices with independent standard Gaussian entries and $X = \tilde A\tilde B^\top + 10\tilde X$ where $\tilde X$ is $500\times 500$ with standard Gaussian entries.}
\end{figure}

While the alternating directions method does converge, as shown in Figure~\ref{fig1}(right panel) it has slow convergence even when $X$ is approximately low-rank .  Hastie et al. noticed that the Algorithm \ref{alg1} looks like a power iteration method, since at each step we multiply the current $A$ or $B$ by either $X$ or $X^\top$ respectively \cite{Hastie}.  Thus, motivated by the idea of orthogonal power iteration, Hastie et al. introduced the idea of using an SVD between each alternation in order to orthogonalize the columns of $A$ and $B$.  Notice that $A$ and $B$ are $m\times r$ and $n\times r$ respectively, so for $r\ll \min\{m,n\}$ these SVDs will often be computable even when the full SVD of $X$ is impractical.  These insights led Hastie et al. to introduce Algorithm \ref{alg2} in \cite{Hastie}.  The authors in \cite{Hastie} suggested that the approach used to show convergence of orthogonal power iteration (see for example \cite{Golub} Theorem 8.2.2, also \cite{HAntil_DChen_SField_2018a}) could be applied to Algorithm \ref{alg2}.  In Section \ref{vecConvergence} we will confirm that the method of \cite{Golub} can indeed be adapted to show convergence of the singular vectors.  However, a more detailed analysis is required to show convergence of the singular values, as we will show in Section \ref{valConvergence}.  Moreover, Algorithm \ref{alg2} can fail to converge or converge to a non-optimal stationary point due to a subtle issue involving the non-uniqueness of the SVD.

\hspace{-13pt}\begin{minipage}{0.48\linewidth}
\begin{algorithm}[H] 
\centering
\caption{Rank-Restricted Soft SVD \cite{Hastie} } \label{alg2}
\begin{algorithmic} 
	\State {\bf Inputs:} An $n\times m$ matrix $X$, \\ \quad\quad\quad\quad Rank restriction $r$, and \\ \quad\quad\quad\quad Regularization parameter $\lambda$
	\State {\bf Outputs:} An $n\times r$ matrix $A$ and \\ \quad\quad\quad\quad\quad An $m\times r$ matrix $B$
	\State
    \State Initialize $D = I_{r \times r}$
    \State Initialize $U \in \mathbb{R}^{n\times r}$ a random \\ \quad orthonormal matrix
    \State Initialize $A = U D$  and $A_{\rm p} = B_{\rm p} = 0$
    \While{$\frac{||A-A_{\rm p}||_{\max}}{||A||_{\max}} + \frac{||B-B_{\rm p}||_{\max}}{||B||_{\max}} > {\rm tol}$}
            \State Set $A_{\rm p} = A$, $B_{\rm p} = B$
            \State Update $B$ leaving $A$ fixed:
            \Indent
            \State $B \leftarrow X^\top A(D^2 + \lambda I_{r\times r})^{-1}$
            \State Find the SVD: $BD = USV^\top$
            \State $D \leftarrow S^{\frac{1}{2}}$
            \State $B \leftarrow UD$
            \EndIndent
            \State Update $A$ leaving $B$ fixed:
            \Indent
            \State $A \leftarrow X B(D^2 + \lambda I_{r\times r})^{-1}$
            \State Find the SVD: $AD = USV^\top$
            \State $D \leftarrow S^{\frac{1}{2}}$
            \State $A \leftarrow UD$
            \EndIndent

    \EndWhile
\end{algorithmic}
\end{algorithm}
\end{minipage}\hspace{5pt}
\begin{minipage}{0.48\linewidth}
\begin{algorithm}[H] 
\centering
\caption{Modified Rank-Restricted Soft SVD \vspace{2.5pt} } \label{alg3}
\begin{algorithmic} 
	\State {\bf Inputs:} An $n\times m$ matrix $X$, \\ \quad\quad\quad\quad Rank restriction $r$, and \\ \quad\quad\quad\quad Regularization parameter $\lambda$
	\State {\bf Outputs:} An $n\times r$ matrix $A$ and \\ \quad\quad\quad\quad\quad An $m\times r$ matrix $B$
	\State
    \State Initialize $D = I_{r \times r}$
    \State Initialize $U \in \mathbb{R}^{n\times r}$ a random \\ \quad orthonormal matrix
    \State Initialize $A = U D$  and $A_{\rm p} = B_{\rm p} = 0$
    \While{$\frac{||A-A_{\rm p}||_{\max}}{||A||_{\max}} + \frac{||B-B_{\rm p}||_{\max}}{||B||_{\max}} > {\rm tol}$}
            \State Set $A_{\rm p} = A$, $B_{\rm p} = B$
            \State Update $B$ leaving $A$ fixed:
            \Indent
            \State $B \leftarrow X^\top A(D^2 + \lambda I_{r\times r})^{-1}$
            \State Find the SVD: $BD = USV^\top$
            \State $D \leftarrow S^{\frac{1}{2}}$,  $W = \text{diag(sign}(V^\top \vec 1))$
            \State $B \leftarrow UWD$
          
            \EndIndent
            \State Update $A$ leaving $B$ fixed:
            \Indent
            \State $A \leftarrow X B(D^2 + \lambda I_{r\times r})^{-1}$
            \State Find the SVD: $AD = USV^\top$
            \State $D \leftarrow S^{\frac{1}{2}}$,  $ W=\text{diag(sign(}V^\top \vec 1))$
          
            \State $A \leftarrow U W D$
            \EndIndent
    \EndWhile
\end{algorithmic}
\end{algorithm}
\end{minipage}

\subsection{Proposed Algorithm}

Despite the similarity of Algorithm \ref{alg2} to orthogonal power iteration, there is a key difference which can cause Algorithm \ref{alg2} to fail to converge.  Orthogonal power iteration uses the QR factorization, which is naturally unique when you specify that the the diagonal entries of $R$ are non-negative.  The SVD on the other hand does not have a natural choice of sign for the singular vectors \cite{Bro}.  The SVD is only unique up to a choice of sign since for any matrix $W$ which is diagonal with diagonal entries in $\{-1,1\}$ we have,
\[ USV^\top = UWSWV^\top = \tilde U S \tilde V^\top. \]
This non-uniqueness means that many SVD algorithms will return different choices of $W$ each time they are run (due to random initialization).  This can lead to failure of Algorithm \ref{alg2} to converge, simply due to oscillations in $A$ and $B$ caused by varying implicit choices of $W$ in the SVD steps.  Moreover, as we will show in Section \ref{signs}, the different choices of $W$ correspond to alternate stationary points of the cost function in \eqref{rrss}.  

To address these issues, we introduce Algorithm \ref{alg3} which is a modification of Algorithm \ref{alg2}.  The new aspect of Algorithm \ref{alg3} is that, after each SVD, we make a unique choice of sign for the left singular vectors.  This seemingly minor addition proves critical for convergence as shown in Figure~\ref{fig2} and as we will prove analytically in Section \ref{valConvergence} below.  In fact, we will show that this choice of sign insures that the matrices $V$ of right singular vectors converge to the identity matrix and that this choice is required to obtain the optimal solution of \eqref{rrss}.

\begin{figure}[h]
\includegraphics[width=0.45\linewidth]{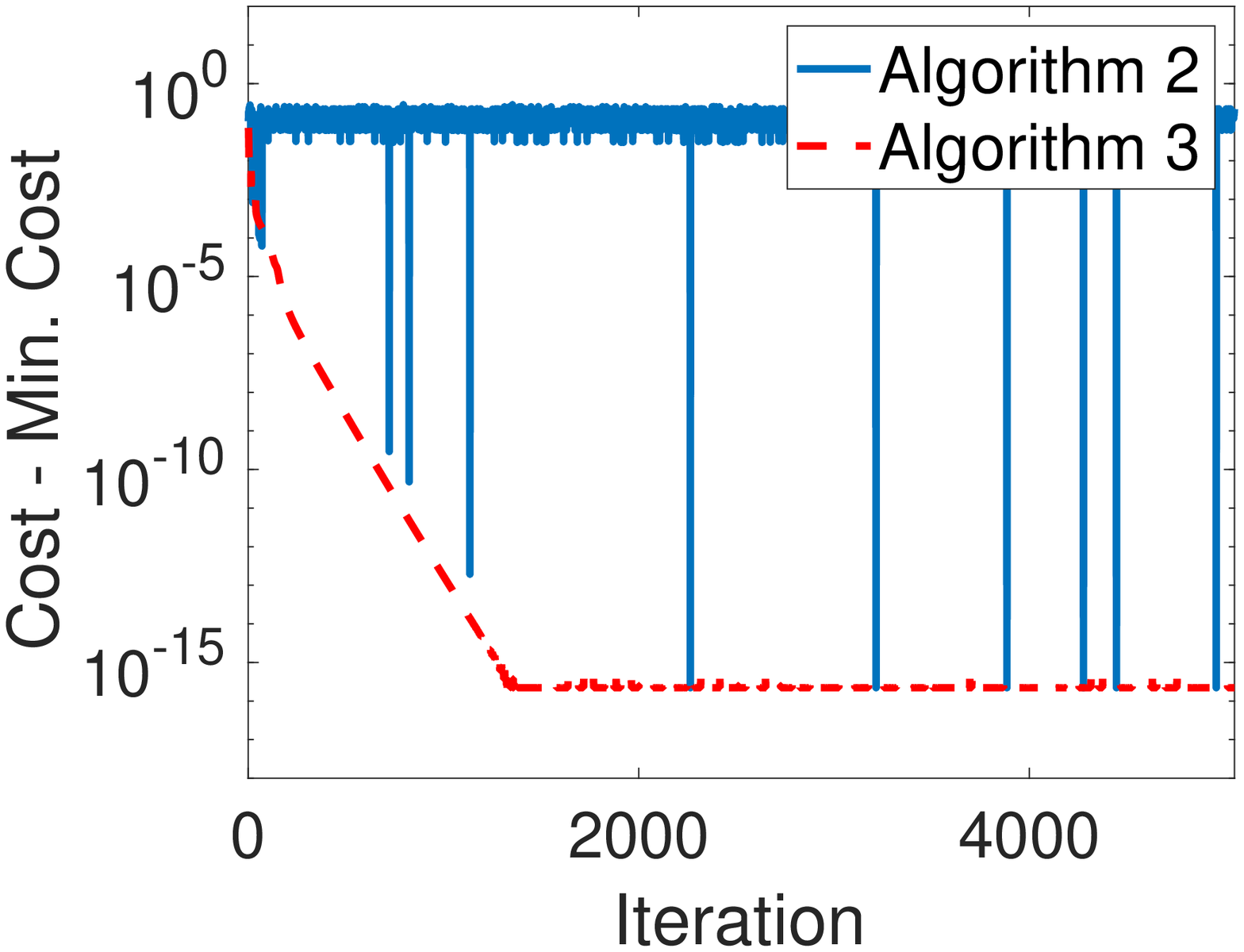}\hspace{10pt}\includegraphics[width=0.45\linewidth]{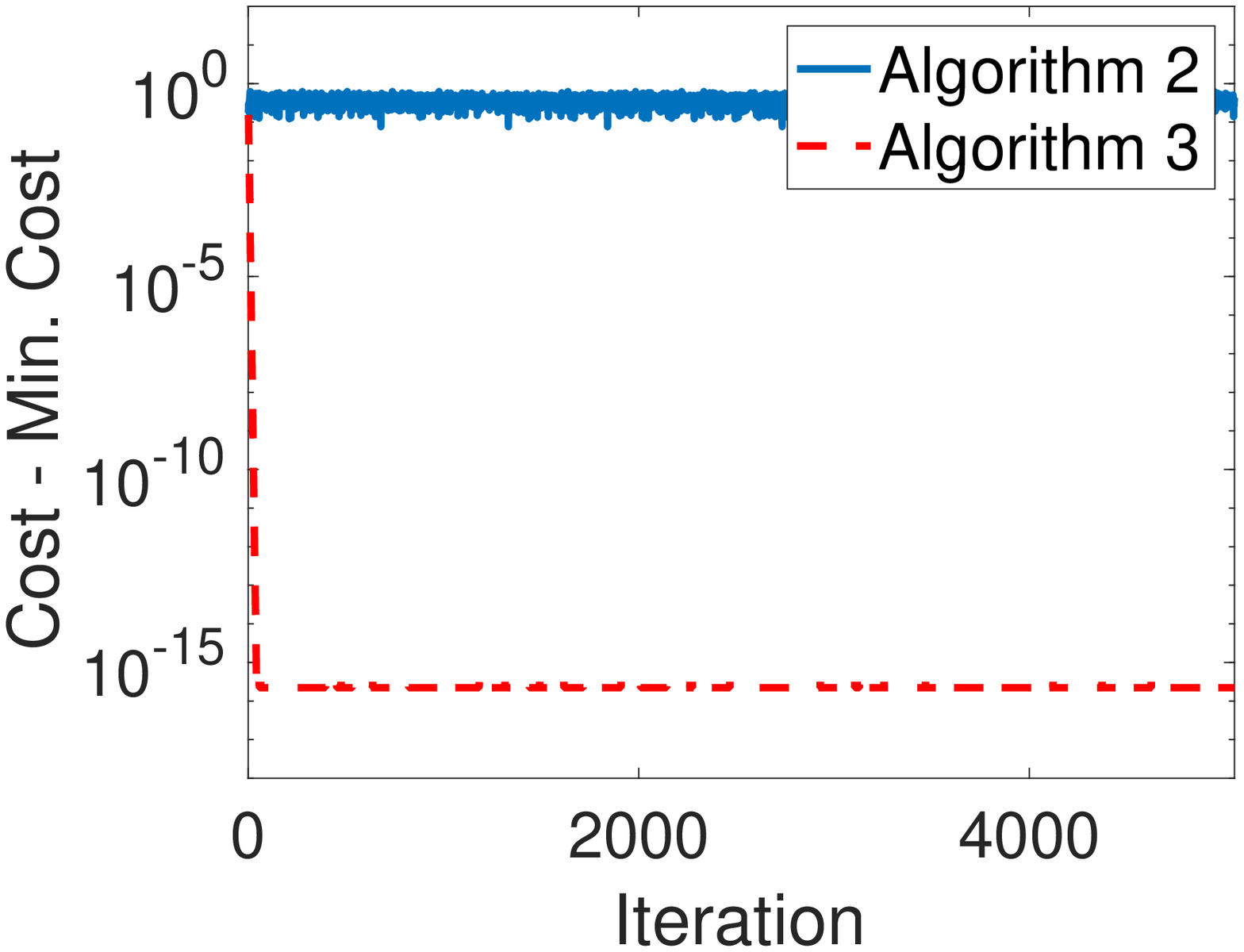}
\caption{\label{fig2} Comparison of Algorithm \ref{alg2} from \cite{Hastie} with our new Algorithm \ref{alg3} on the same full-rank (left) and approximately low-rank (right) examples from Figure~\ref{fig1}.}
\end{figure}

We will formalize Algorithm \ref{alg3} mathematically since Algorithm \ref{alg2} can then be obtained by simply redefining the choice of $W$.  Based on Algorithm \ref{alg3} we make the following recursive definitions,
\begin{subequations}
\begin{align}
B_{k+1} &= X^\top U_k W_k D_k(D_k^2 + \lambda I)^{-1}  						\label{Beq} \\
\tilde U_{k}\tilde W_{k}\tilde D_{k}^2\tilde W_{k}\tilde V_{k}^\top &= B_{k+1}D_{k}  \label{svd1} \\
A_{k+1} &= X\tilde U_k\tilde W_k \tilde D_k(\tilde D_k^2 + \lambda I)^{-1} 				\label{Aeq} \\
 U_{k+1} W_{k+1} D_{k+1}^2 W_{k+1} V_{k+1}^\top &= A_{k+1}\tilde D_{k}		\label{svd2}
\end{align}
\end{subequations}
where \eqref{svd1} and \eqref{svd2} define all the quantities on the left hand side by computing the SVD of the right hand side.  We initialize $\tilde D_{-1} = D_0 = W_0 = I$ and choose $U_0$ to be a random orthonormal $n\times r$ matrix and set $A_0 = U_0 W_0 D_0$.   

The matrices $W_k, \tilde W_k$ are diagonal matrices where each diagonal entry is either $1$ or $-1$.  These matrices define the choice of signs for the left and right singular vectors resulting from the SVD computation.  In fact, due to random initializations of most SVD algorithms, the matrices $W_k,\tilde W_k$ are typically random and will be different each time the SVD algorithm is run.  As we will see, this will be the cause of the erratic behavior of the cost function in Algorithm \ref{alg2} as shown in Figure~\ref{fig2}.

A more concise iteration can be obtained by solving solving \eqref{svd2} (at the previous step) for $U_{k} W_{k} D_{k} = A_k \tilde D_{k-1} V_k W_k D_k^{-1}$ and substituting into \eqref{Beq} we have,
\begin{align} \label{Beq2} B_{k+1} = X^\top A_{k}\tilde D_{k-1} V_{k} W_{k}D_k^{-1}(D_k^2 + \lambda I)^{-1}. \end{align}
Similarly, solving \eqref{svd1} for $\tilde U_k \tilde W_k\tilde D_k = B_{k+1} D_k \tilde V_k \tilde W_k  \tilde D_k^{-1}$ and by substituting into \eqref{Aeq} we can write,
\begin{align}\label{Aeq2}  A_{k+1} = X B_{k+1}D_k \tilde V_k \tilde W_k \tilde D_k^{-1}(\tilde D_k^2 + \lambda I)^{-1}. \end{align}
Here we can immediately see that the product $A_{k+1}B_{k+1}^\top$ will not converge unless the signed right singular vectors $\tilde V_k \tilde W_k, W_{k}V_{k}^\top$ of \eqref{svd1},\eqref{svd2} converge since,
\[ A_{k+1}B_{k+1}^\top = X B_{k+1}D_k \tilde V_k \tilde W_k \tilde D_k^{-1}(\tilde D_k^2 + \lambda I)^{-1}(D_k^2 + \lambda I)^{-1}D_k^{-1}W_{k}V_{k}^\top \tilde D_{k-1}A_{k}^\top X. \]
This explains the jumps of Algorithm \ref{alg2} shown in Figure~\ref{fig2}.  

\subsection{Overview}

In Section \ref{vecConvergence} we will show that, in an appropriate sense, we have $U_k\to U$ and $\tilde U_k \to V$.  Then, in Section \ref{valConvergence}, we turn to the singular values and show that $D_k,\tilde D_k$ both converge to $I_{r\times p} D I _{p\times r}$ given by the softmax function $D = \sqrt{(S-\lambda I)^+}$. Finally, in Section \ref{signs} we will show that $V_k,\tilde V_k$ converge to diagonal matrices determined by the choice of $W_k,\tilde W_k$. We will see that any convergent choice for the diagonal sign matrices $W_k,\tilde W_k$ will yield a convergent algorithm.  These results will culminate in Theorem \ref{finalthm} which reveals that, assuming $\tilde W_k \to \tilde W_*$ and $W_k \to W_*$, we have the limiting matrices, 
\begin{align*}
A_k &\to A_* = US D(D^2 + \lambda I)^{-1} I_{p\times r} \tilde W_* \\
B_k &\to B_* = VS D(D^2+\lambda I)^{-1} I_{p\times r} W_*
\end{align*}
for Algorithm \ref{alg3}.  Moreover, the dependence of the first term of the cost function \eqref{rrss} on the sign matrices is given by,
\begin{align}\label{mincost} ||X-A_*B_*^\top||_F =  ||S - S^2 D^2 (D^2+\lambda I)^{-2} I_{p\times r} \tilde W_* W_*  I_{r\times p}||_F \end{align}
 and only the choice $\tilde W_* W_* = I$ will minimize the cost.  When $\lambda < S_{rr}$ the above cost simplifies to,
 \[ ||X-A_*B_*^\top||_F =  ||S - (S-\lambda)^+ I_{p\times r} \tilde W_* W_*  I_{r\times p}||_F \]
 which is optimal when $\tilde W_* W_* = I$.  This explains the large cost values for Algorithm \ref{alg2} shown in Figure~\ref{fig2} since the random $W_k,\tilde W_k$ essentially replace $\tilde W_*, W_*$ with random sign matrices.  Of course, occasionally these random sign matrices yield $\tilde W_k W_k = I$, which explains why the cost sometimes jumps down to the optimal cost.  This also justifies our choice in Algorithm \ref{alg3} where $W_k,\tilde W_k$ are chosen to insure that the sum of the columns of $W_k V_k$ and $\tilde W_k \tilde V_k$ are positive.  As $V_k,\tilde V_k$ converge to diagonal matrices, this choice will guarantee that $\tilde W_* W_* = I$, thereby obtaining the minimal cost solution.

\section{Convergence of the Singular Vectors}\label{vecConvergence}

The first part of proving the convergence of Algorithm \ref{alg3} is showing that the sequences $U_k$ and $\tilde U_k$ defined in \eqref{svd1} and \eqref{svd2} converge to the top $r$ left and and right singular vectors of $X$ respectively.  In other words, if $X=USV^\top$ is the SVD of $X$ then loosely speaking we have $U_k \to U_{(1:r)}$ and $V_k\to V_{(1:r)}$ where the subscript $(1:r)$ indicates the first through $r$-th columns of the matrix.  The reason we say `loosely speaking' is due to the non-uniqueness of sign in the singular vectors, even for unique singular values (for repeated singular values we only have uniqueness up to orthogonal linear transformations).  Thus, the first column of $U_k$ could alternate between that of $U$ and its negative and this would still be considered convergence since we would have obtained the correct subspace.  

We define convergence in terms of the norm of the matrix of inner products $||U_k^\top U_{(r+1:n)}||$ converging to $0$ (any matrix norm can be used since this always implies $U_k^\top U_{(r+1:n)}$ is zero).  Since $U_k U_k^\top = I_{r\times r}$, the columns of $U_k$ span an $r$-dimensional subspace, so if $U_k^\top U_{(r+1:n)} = 0$ this subspace is orthogonal to the subspace spanned by the last $n-r$ columns of $U$.   Thus, $||U_k^\top U_{(r+1:n)}||_{\max}\to 0$ implies that the subspace spanned by the columns of $U_k$ is aligning with the subspace spanned by the first $r$ columns of $U$.  As shown in Figure~\ref{fig3} we have $||U_k^\top U_{(r+1:n)}||_{\max}\to 0$ for both Algorithm \ref{alg2} and Algorithm \ref{alg3}.  

In this section we will prove that this convergence is independent of the choice of $W_k,\tilde W_k$ and show that the convergence rate is determined by the ratio of the $(r+1)$-st and $r$-th squared singular values of $X$.  In particular, when $X$ is low-rank or approximately low-rank, this will imply the fast convergence observed in Figure~\ref{fig1}.  We first note that the iteration \eqref{Beq}-\eqref{svd2} is rank preserving in the generic case when $X$ is full-rank.

\begin{figure}[h]
\includegraphics[width=0.45\linewidth]{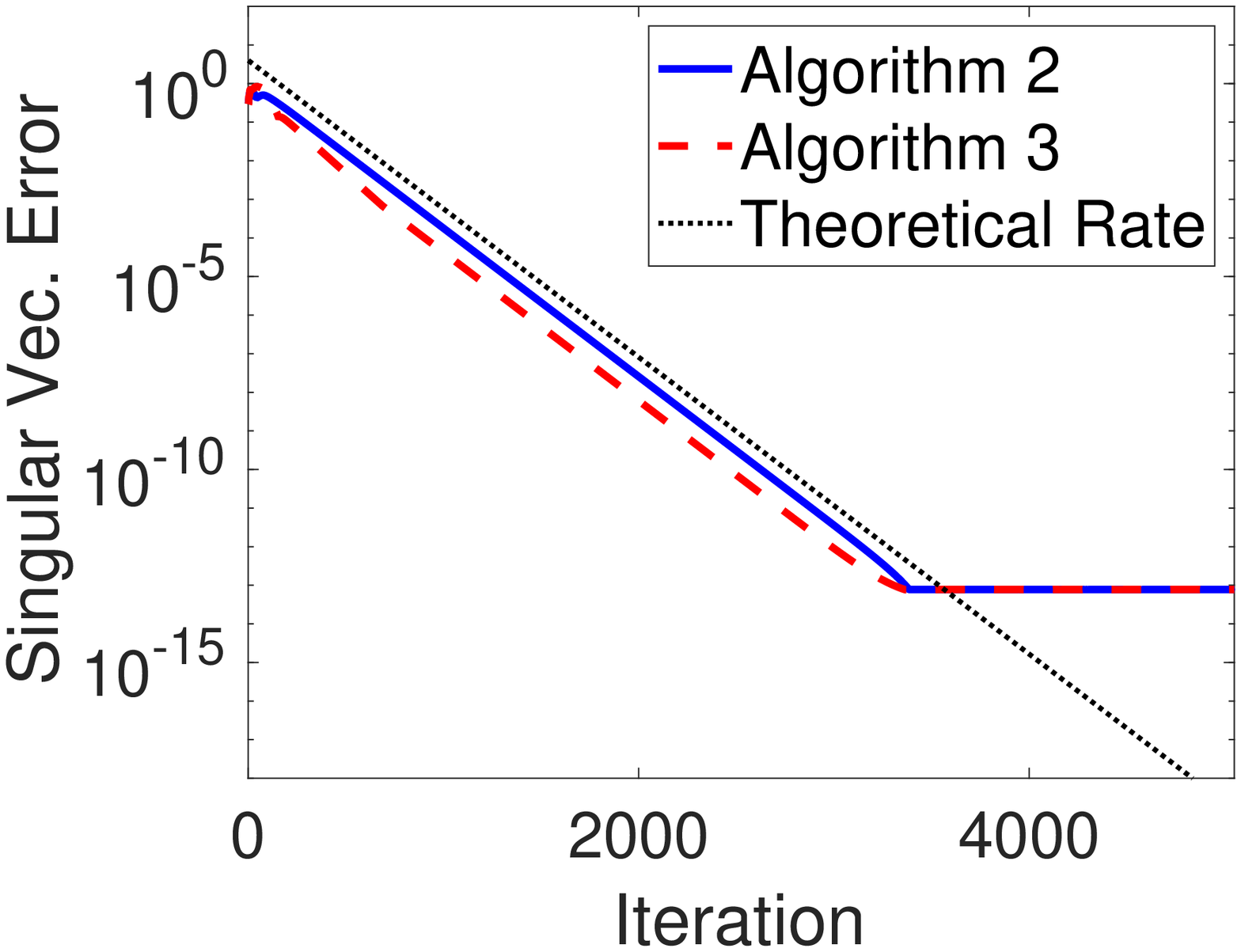}\hspace{10pt}\includegraphics[width=0.45\linewidth]{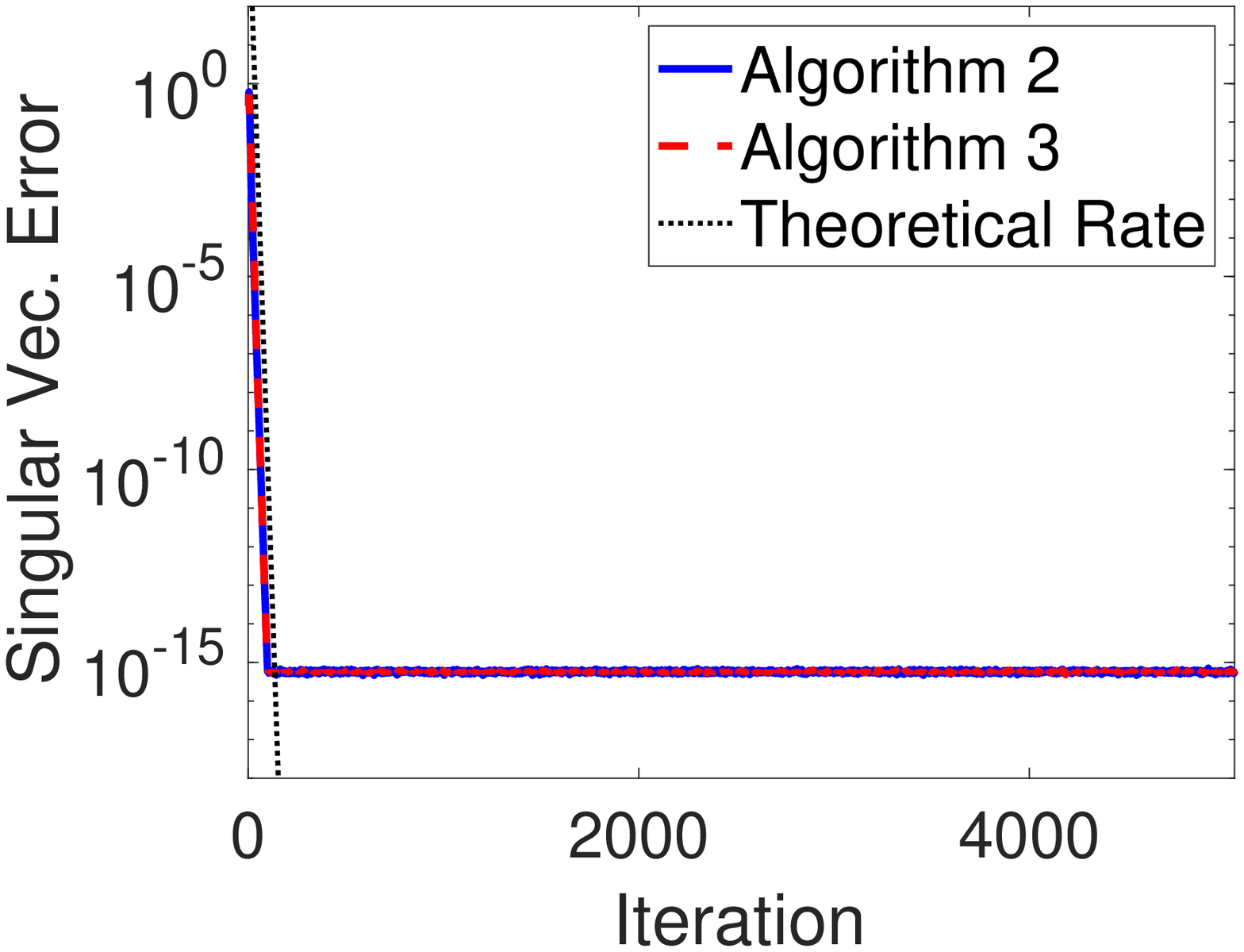}
\caption{\label{fig3} Comparison of the convergence of the singular vectors on the same full-rank (left) and approximately low-rank (right) examples from Figure~\ref{fig1}.  Error is measured by $||U_k^\top U_{(r+1:n)}||_{\max}$, where $U_{(r+1:n)}$ is the matrix containing the $(r+1)$-st through $n$-th columns of $U$.  The theoretical convergence rate $\left(\frac{s_{r+1}}{s_r}\right)^2$ shown is proven in Theorem \ref{thm1} .  Notice that the singular vectors converge for both Algorithm \ref{alg2} from \cite{Hastie} and our new Algorithm \ref{alg3} .}
\end{figure}

\begin{lemma} \label{FullRank1}
 Let $X \in \mathbb{R}^{n\times m}$, have full rank, namely ${\rm rank}(X) = \min\{m,n\}$, then for all $k$ the matrices $A_{k}, B_{k}, U_k, W_k, D_k, V_k, \tilde U_k, \tilde W_k, \tilde D_k,\tilde V_k$ defined by the iteration \eqref{Beq}-\eqref{svd2} are all full rank. 
\end{lemma}
\begin{proof}
The algorithm is initialized with $A_0 = U_0 W_0 D_0$, where $U_0$ is a random matrix and thus generically full rank and $D_0 = W_0 = I$ is full rank.  By \eqref{Beq} we have $B_{k+1} = X^\top A_k(D_k^2+\lambda I)^{-1}$ and since $X$ and $D_k^2+\lambda I$ are full rank, we have ${\rm rank}(B_{k+1}) = {\rm rank}(A_k)$.  This establishes the base case, and if we inductively assume $A_k,D_k$ are full rank we immediately find that $B_{k+1}$ is full rank and thus $B_{k+1}D_k$ is also full rank.  Since the right-hand-side of \eqref{svd1} is full rank, all the matrices $\tilde U_k,\tilde W_k,\tilde D_k,\tilde V_k$ on the left-hand-side of \eqref{svd1} are full rank since they are defined to be the SVD of a full rank matrix.  By \eqref{Aeq} we have $A_{k+1}$ written as a product of full rank matrices and thus $A_{k+1}$ is full rank.  Finally, the right-hand-side of \eqref{svd2} is now full rank which implies that all the matrices on the left-hand-side, $U_{k+1},W_{k+1},D_{k+1},V_{k+1}$ are all full rank.  This completes the induction.
\end{proof}

When $X$ is not full rank, generically the random initial matrix $U_0$ will not be orthogonal to the subspace spanned by the rows of $X$ and since $B_1 = X^\top A_0/(1+\lambda)$ we find ${\rm rank}(B_1) = \min\{{\rm rank}(X),{\rm rank}(A_0)\}$.  Note that since $D_0=I$ we have $(D_0^2+\lambda I)^{-1} = I/(1+\lambda)$.  When ${\rm rank}(X)\geq r$ we expect all of the matrices in Lemma \ref{FullRank1} to have rank $r$ and when ${\rm rank}(X)<r$ they should all have rank equal to ${\rm rank}(X)$.  However, showing that $U_k$ does not evolve to become orthogonal to the span or the rows of $X$ requires Theorem \ref{thm1} below.  

The next step is to make a connection between the iteration \eqref{Beq}-\eqref{svd2} and the SVD of $X$.  In the next lemma we show how the \eqref{Beq} followed by \eqref{Aeq} is related to multiplication by $XX^\top$ and similarly \eqref{Aeq} followed by \eqref{Beq} is related to multiplication by $X^\top X$.

\begin{lemma}\label{products} Let $X \in \mathbb{R}^{n\times m}$, and using the notation of \eqref{Beq}-\eqref{svd2} define
\begin{align} 
P_{k+1} &\equiv D_{k+1}^2 V_{k+1}^\top (\tilde D_k^2 + \lambda I)\tilde W_k \tilde V_k^\top D_k^{-2} (D_k^2 + \lambda I) W_k \nonumber  \\
\tilde P_{k+1} &\equiv \tilde  D_{k+1}^2 \tilde V^\top_{k+1} (D_{k+1}^2 +\lambda I) W_{k+1} V^\top_{k+1} \tilde D_k^{-2} (\tilde D_k ^{2} +\lambda I ) \tilde W_k \nonumber
\end{align}
then 
\begin{align} 
XX^\top U_k &= U_{k+1} P_{k+1} &\hspace{20pt} (XX^\top)^k U_0 &= U_k P_k \cdots P_1 \nonumber \\
X^\top X \tilde U_k &= \tilde U_{k+1} \tilde P_{k+1} &\hspace{20pt} (X^\top X)^k \tilde U_0 &= \tilde U_k \tilde P_k \cdots \tilde P_1 \nonumber
\end{align}
and the products
\begin{align} 
Q_k  &\equiv D_{k}^{-2}P_k \cdots P_1 \nonumber \\ &=   V_{k}^\top \left(\prod_{i=1}^{k-1} (\tilde D_i^2 + \lambda I)\tilde W_i \tilde V_i^\top (D_i^2 + \lambda I) W_i V_{i}^\top \right) (\tilde D_0^2 + \lambda I)\tilde W_0 V_0^\top (1+\lambda) \nonumber  \\
\tilde Q_k &\equiv \tilde D_{k}^{-2} \tilde P_k \cdots \tilde P_1 = \tilde V^\top_{k} (D_{k}^2 +\lambda I) W_{k} Q_{k}  \nonumber 
\end{align}
are invertible with inverses bounded by $||Q_k^{-1}|| \leq \lambda^{1-2k}$, and $||\tilde Q_k^{-1}||  \leq \lambda^{2-2k}$.
\end{lemma}
\begin{proof}
We first solve \eqref{Beq} for $X^\top U_k = B_{k+1}(D_k^2 + \lambda I)D_k^{-1} W_k$ to obtain, 
\begin{align} XX^\top U_k &= X B_{k+1}(D_k^2 + \lambda I)D_k^{-1} W_k  \nonumber \\ 
& = A_{k+1}(\tilde D_k^2 + \lambda I)\tilde D_k \tilde W_k \tilde V_k^\top D_k^{-1} (D_k^2 + \lambda I)D_k^{-1} W_k  \nonumber\\
&= U_{k+1} D_{k+1}^2 V_{k+1}^\top (\tilde D_k^2 + \lambda I)\tilde W_k \tilde V_k^\top D_k^{-2} (D_k^2 + \lambda I) W_k
\end{align}
where the second equality follows from \eqref{Aeq2} and the last follows from \eqref{svd2} after rearranging the diagonal matrices.  The definition of $P_k$ then immediately yeilds $XX^\top U_k = U_{k+1} P_{k+1}$ and a similar computation shows $XX^\top \tilde U_k = \tilde U_{k+1} \tilde P_{k+1}$.

The formulas for $Q_k$ and $\tilde Q_k$ follow by a simple induction using the formulas for $P_k,\tilde P_k$.  Note that $Q_k,\tilde Q_k$ are products of diagonal matrices (with non-zero diagonal entries), sign matrices and orthogonal matrices and thus are both invertible.  Moreover, since $\lambda>0$ we have the upper bound,
\[ ||Q_k^{-1}|| \leq \left(\prod_{i=1}^{k-1} \frac{1}{||\tilde D_i^2 + \lambda I || \, ||D_i^2 + \lambda I ||} \right) \frac{1}{||\tilde D_{0}^2 + \lambda I || (1+\lambda)} \leq \lambda^{1-2k} \]
 and $||\tilde Q_k^{-1}|| \leq \frac{||Q_k^{-1}||}{||D_k^2 +\lambda I||} \leq \lambda^{2-2k}$.
\end{proof}

In order to connect the iteration \eqref{Beq}-\eqref{svd2} to the singular vectors of $X$ we will use the formulas,
\[ (XX^\top)^k U_0 = U_k D_k^2 Q_k, \hspace{30pt} (X^\top X)^k \tilde U_0 = \tilde U_k \tilde D_k^2 \tilde Q_k  \]
which follow from Lemma \ref{products}.  Substituting the SVD of $X = USV^\top$ results in,
\[ US^{2k} U^\top U_0 = U_k D_k^2 Q_k, \hspace{30pt} VS^{2k} V^\top \tilde U_0 = \tilde U_k \tilde D_k^2 \tilde Q_k \]
and using the invertibility of the $D_k,\tilde D_k,Q_k,\tilde Q_k$ matrices we have,
\begin{align}\label{Uiter} 
U^\top U_k = S^{2k} U^\top U_0 D_k^{-2} Q_k^{-1} , \hspace{30pt}   V^\top \tilde U_k = S^{2k} V^\top \tilde U_0 \tilde D_k^{-2} \tilde Q_k^{-1}.  
\end{align}
Notice that we have again rearranged the diagonal matrices.


The key to leveraging \eqref{Uiter} for analyzing the convergence of $U_k,\tilde U_k$ is to split the true singular vectors, $U$, into two groups by choosing an arbitrary $\ell \in \{1,...,p-1\}$ where $p=\min\{m,n\}$.  We then split $U = [U_{(1)} \, U_{(2)}]$ where $U_{(1)}$ contains the first $\ell$ columns of $U$, and similarly $V = [V_{(1)} \, V_{(2)}]$ and finally we split the diagonal matrix of singular values as $S = \left(\begin{array}{cc} S_1 & 0 \\ 0 & S_2 \end{array}\right)$ where $S_1$ is $\ell \times \ell$ and contains the first $\ell$ singular values.  

\begin{theorem}\label{thm1} Let $X \in \mathbb{R}^{n\times m}$ have SVD $X = US V^\top$ and set $p=\min\{m,n\}$ then, using the notation of Lemma \ref{products}, for any splitting of the singular vectors $\ell \in \{1,...,p-1\}$ we have
\begin{align} U_{(1)}^\top U_{k,\ell} &= S_1^{2k} U_{(1)}^\top U_{0,\ell} Z_{k,\ell} &\hspace{20pt} U_{(2)}^\top U_{k,\ell} &= S_2^{2k} U_{(2)}^\top U_{0,\ell} Z_{k,\ell} \label{Usplit} \\
V_{(1)}^\top \tilde U_{k,\ell} &= S_1^{2k} V_{(1)}^\top \tilde U_{0,\ell} \tilde Z_{k,\ell} &\hspace{20pt} V_{(2)}^\top \tilde U_{k,\ell} &= S_2^{2k} V_{(2)}^\top \tilde U_{0,\ell} \tilde Z_{k,\ell} \label{Vsplit}
\end{align}
where $U_{k,\ell}, \tilde U_{k,\ell}$ are the first $\ell$ columns of $U_k,\tilde U_k$ respectively and $Z_{k,\ell},\tilde Z_{k,\ell}$ are the first $\ell$ rows of $D_k^{-2}Q_k^{-1},\tilde D_k^{-2} \tilde Q_k^{-1}$ respectively.  Moreover, as $k \rightarrow \infty$, we have
\[ \frac{||U_{(2)}^\top U_{k,\ell} ||}{||U_{(1)}^\top U_{k,\ell} ||} \leq c_\ell \left(\frac{s_{\ell+1}}{s_\ell} \right)^{2k} \to 0  \hspace{40pt}  \frac{||V_{(2)}^\top \tilde U_{k,\ell} ||}{||V_{(1)}^\top \tilde U_{k,\ell} ||} \leq \tilde c_\ell \left( \frac{s_{\ell+1}}{s_\ell} \right)^{2k} \to 0. \]
\end{theorem}

\begin{proof} From \eqref{Uiter} we have,
\[ \left(\begin{array}{c} U_{(1)}^\top \\ \vspace{3pt} U_{(2)}^\top \end{array}\right) U_k = \left(\begin{array}{cc} S_1^{2k} & 0 \\ 0 & S_2^{2k} \end{array}\right) \left(\begin{array}{c} U_{(1)}^\top \\ \vspace{3pt} U_{(2)}^\top \end{array}\right) U_0 D_k^{-2} Q_k^{-1} \]
which immediately splits into the equations \eqref{Usplit} and a similar splitting occurs for $V$ which yields \eqref{Vsplit}.  Next we solve the left equation of \eqref{Usplit} for $Z_{k,\ell}$ and substitute into the right equation of \eqref{Usplit} to find,
\[ U_{(2)}^\top U_{k,\ell} = S_2^{2k} U_{(2)}^\top  U_{0,\ell} (U_{(1)}^\top U_{0,\ell})^{-1} S_1^{-2k} U_{(1)}^\top U_{k,\ell} \]
and obtain the upper bound, 
\[ ||U_{(2)}^\top U_k|| \leq ||S_2^{2k}|| \, c_\ell \, ||S_1^{-2k}|| \, ||U_{(1)}^\top U_k|| = \left(\frac{s_{\ell+1}}{s_\ell} \right)^{2k} ||U_{(1)}^\top U_k|| \]
 where the constant $c_\ell$ is determined by the inner products with $U_{0,\ell}$ and is independent of $k$.  
\end{proof}

The power of Theorem \ref{thm1} is that the splitting $\ell$ was arbitrary.  In the generic case of distinct singular values, $\ell=1$ immediately implies that the first column of $U_k$ becomes orthogonal to the last $p-1$ left singular vectors of $X$ (columns of $U$) and hence must lie in the space spanned by the first left singular vector of $X$.  Then, $\ell=2$ implies that the second column of $U_k$ must be orthogonal to the last $p-2$ left singular vectors.  Moreover, the definition of $U_k$ via the SVD in \eqref{svd2} implies that the second column of $U_k$ is orthogonal to the first column of $U_k$ and hence must be in the subspace spanned by the second left singular vector of $X$.  Inductively, this shows that the columns of $U_k$ converge to lie in the subspaces spanned by the corresponding columns of $U$.  In the generic case of distinct singular values, this means that the columns of $U_k$ are converging to those of $U$ up to sign.  Moreover, in the non-generic case of a repeated singular value, Theorem \ref{thm1} shows the convergence of the corresponding columns of $U_k$ to the subspace spanned by the singular vectors corresponding to the repeated singular value.  We can now turn to the convergence of the singular values.

\section{Convergence of the Singular Values}\label{valConvergence}

We can combine \eqref{Beq} and \eqref{svd1} into a single equation (and similarly for \eqref{Aeq} and \eqref{svd2}),
\begin{align}
\tilde U_{k}\tilde W_{k}\tilde S_{k}\tilde W_{k}\tilde V_{k}^\top  &= X^\top U_k W_k S_k(S_k + \lambda I)^{-1}     \label{Seq}  \\
U_{k+1} W_{k+1} S_{k+1} W_{k+1} V_{k+1}^\top  &= X\tilde U_k\tilde W_k \tilde S_k(\tilde S_k + \lambda I)^{-1}  \label{tSeq}
\end{align} 
where $S_k = D_k^2$ and $\tilde S_k = \tilde D_k^2$ and the terms on the left-hand-side of \eqref{Seq} and \eqref{tSeq} are defined to be the singular value decomposition of the right-hand-side.  Substituting the singular value decomposition of $X = USV^\top$ we have,
\begin{align}
\tilde U_{k}\tilde W_{k}\tilde S_{k}\tilde W_{k}\tilde V_{k}^\top  &= VSU^\top U_k W_k S_k(S_k + \lambda I)^{-1}     \label{Seq2}  \\
U_{k+1} W_{k+1} S_{k+1} W_{k+1} V_{k+1}^\top  &= USV^\top \tilde U_k\tilde W_k \tilde S_k(\tilde S_k + \lambda I)^{-1}.  \label{tSeq2}
\end{align} 
We first consider the simplified iteration where the singular vectors are set equal to their limits, namely, $U_k=U_{(1:r)}$ and $\tilde U_k = V_{(1:r)}$.  Since $U_k \to U_{(1:r)}$ and $\tilde U_k \to V_{(1:r)}$ we will be able to use a perturbation argument to extend this simplified case to the true $U_k, \tilde U_k$ sequences.  In the simplified iteration, $U^\top U_k = V^\top \tilde U_k = I_{n\times r}$ where $I_{n\times r}$ is an $r$-by-$r$ identity matrix concatenated with an $(n-r)$-by-$r$ matrix of all zeros.  In this case we obtain 
\begin{align}
\tilde U_{k}\tilde W_{k}\tilde S_{k}\tilde W_{k}\tilde V_{k}^\top  &= V I_{n\times r} W_k S S_k(S_k + \lambda I)^{-1}     \label{Seq3}  \\
U_{k+1} W_{k+1} S_{k+1} W_{k+1} V_{k+1}^\top  &= U I_{n\times r} \tilde W_k S \tilde S_k(\tilde S_k + \lambda I)^{-1}.  \label{tSeq3}
\end{align} 
Note that the left-hand-sides of \eqref{Seq3} and \eqref{tSeq3} are defined to be the unique SVD of the right-hand-sides.  This implies that $\tilde U_k = VI_{n\times r}$ and $U_{k+1} = U I_{n\times r}$ and $\tilde V_k^\top = V_{k+1}^\top = I_{n\times r}$ which shows that this is a fixed point for the singular vectors.  Moreover, we obtain the following iteration for the singular values,
 \begin{align}
\tilde S_{k} &= S S_k(S_k + \lambda I)^{-1}     \label{Seq4}  \\
S_{k+1}   &= S \tilde S_k(\tilde S_k + \lambda I)^{-1}.  \label{tSeq4}
\end{align} 
Since these are all diagonal matrices, we can focus on the fixed point iteration for a single diagonal entry $s_k = (S_k)_{ii}$ and $s = S_{ii}$ we find,
\begin{equation}\label{seq1} s_{k+1} = s^2 \frac{s_k}{s_k + \lambda}\left(\frac{s s_k}{s_k + \lambda} + \lambda\right)^{-1} = \frac{s^2 s_k}{s_k(s + \lambda) + \lambda^2} \end{equation}
for any $i \in \{1,...,r\}$.  

\begin{lemma}\label{softmax}
For any $s,\lambda,s_0 \in \mathbb{R}$ with $s\neq \lambda$ the iteration \eqref{seq1} converges locally to the softmax function,
\[ s_k  \to (s-\lambda)^+ \equiv \max \{0,s-\lambda\} , \]
which is the only stable fixed point.
\end{lemma}
\begin{proof}
The fixed points of this iteration are the solutions $\hat s$ of $\hat s = \frac{s^2\hat s}{\hat s(s+\lambda)+\lambda^2}$ which implies
\[ \hat s(\hat s(s+\lambda) +  \lambda^2 - s^2 ) = 0 \]
so the fixed points are $\hat s = 0$ and $\hat s = s - \lambda$.  Next we analyze the stability of the fixed points by computing the derivative of the iteration,
\[ \frac{d}{ds_k} \left( \frac{s^2 s_k}{s_k(s + \lambda) + \lambda^2} \right) = \frac{ (s_k(s+\lambda) +\lambda^2)s^2 - s^2 s_k(s+\lambda)}{(s_k(s+\lambda)+\lambda^2)^2}  \] 
and evaluating at the fixed point $s_k = \hat s = 0$ we find 
\[ \left. \frac{d}{ds_k} \left( \frac{s^2 s_k}{s_k(s + \lambda) + \lambda^2} \right) \right|_{s_k = 0} = \frac{s^2}{\lambda^2}  \] 
and at the fixed point $s_k = \hat s = s-\lambda$ we find 
\[ \left. \frac{d}{ds_k} \left( \frac{s^2 s_k}{s_k(s + \lambda) + \lambda^2} \right) \right|_{s_k = s-\lambda} = \frac{ ((s-\lambda)(s+\lambda) +\lambda^2)s^2 - s^2 (s-\lambda)(s+\lambda)}{((s-\lambda)(s+\lambda)+\lambda^2)^2}  = \frac{\lambda^2}{s^2} .\] 
Thus we see that when $s < \lambda$ the fixed point $\hat s = 0$ is stable and when $s > \lambda$ the fixed points $\hat s = s-\lambda$ is stable.  In other words, when $s-\lambda$ is positive the stable fixed point is $s-\lambda$ and when $s-\lambda$ is negative the stable fixed point is zero, thus we see that the iteration converges to the soft-max function, 
\[ s_k \to \max\{0,s-\lambda\} \]
This completes the proof.
\end{proof}
The case $\lambda\neq s$ is generic, however, we note that for the case of $s=\lambda$ we have the simplified iteration $s_{k+ 1} =   \frac{\lambda  s_k}{ 2 s_{k} + \lambda}$ and inductively we have, $s _{k}=   \frac{\lambda  s_{0}}{( 2 k )  s_{0} + \lambda }$
so unless $s_0 = -\frac{\lambda}{2k}$ for some $k \in \mathbb{N}$, we again have $s_{k} \to 0 = \max\{0,s-\lambda\}$. 

Lemma \ref{softmax} holds for any real $s \neq \lambda$ and any initial condition $s_0$ including negative numbers.  Of course, in our current application, these are all constrained to be non-negative.  When any of them are zero the iteration is trivial, so in the next lemma we consider the case when $s,\lambda,s_0>0$ and show a stronger convergence property that will be required for the perturbation result.  
\begin{lemma} \label{Est1}
For any $s,\lambda,s_0 \in (0,\infty)$, with $ s \neq \lambda,$ there exists $c \in [0,1)$ such that 
\[ |s_{k+1}-(s-\lambda)^+| \leq c |s_k - (s-\lambda)^+| \]
and the iteration \eqref{seq1} converges globally on $(0,\infty)$ to the softmax function, $s_k  \to (s-\lambda)^+$.
\end{lemma}
\begin{proof} Note that $s,\lambda,s_0>0$ implies $s_k \geq 0$ for all $k$ by a simple induction.  

First consider the case when $\lambda  > s $ so that $(s-\lambda)^+ =0 $.  Setting $c_1 = \frac {s^2}{\lambda ^2} < 1$ we have
\[ | s_{k+1} - (s-\lambda)^+| = \frac{s^2 s_k} {s_k (s+ \lambda ) + \lambda ^2 } <  \frac{s^2}{\lambda^2}   s_k = c_1  |s_k - (s-\lambda)^+|. \]

Next consider the case where $\lambda < s$ so that $(s-\lambda)^+ =  (s-\lambda)$ and
\begin{align}\label{eqa}  (s_{k+1} -(s - \lambda)) =  \frac{s^2 s_k - s_k (s^2-\lambda^2 ) - \lambda ^2(s - \lambda) } {s_k (s+ \lambda ) + \lambda ^2 }  =  \frac{\lambda ^2  } {s_k (s+ \lambda ) + \lambda ^2} (s_k - (s- \lambda)). \end{align}
Since $\frac{\lambda ^2  } {s_k (s+ \lambda ) + \lambda ^2} \leq 1$, \eqref{eqa} implies $|s_{k+1}-(s-\lambda)| \leq |s_k - (s-\lambda)|$ and inductively
\[ |s_{k+1}-(s-\lambda)| \leq |s_0 - (s-\lambda)| \] 
which means that the sequence can never move further away from $s-\lambda$.   Moreover, the sequence can never move to the other side of $s-\lambda$, namely, since $\frac{\lambda ^2  } {s_k (s+ \lambda ) + \lambda ^2} > 0$, if $s_0 \geq s-\lambda$ then \eqref{eqa} implies that $s_0 \geq s_k \geq s-\lambda$ for all $k$, and if $s_0 < s-\lambda$ then $s_0 \leq s_k < s-\lambda$ for all $k$.

Now if $s_0 < s-\lambda$ then we have $s_k \geq s_0$ for all $k$ and setting $c_2 = \frac{\lambda ^2} {s_0 (s+ \lambda ) + \lambda ^2} < 1$, \eqref{eqa} implies,
\[ | s_{k+1} - (s-\lambda)^+| = \frac{\lambda ^2 |s_k - (s- \lambda)| } {s_k (s+ \lambda ) + \lambda ^2} \leq \frac{\lambda ^2 |s_k - (s- \lambda)| } {s_0 (s+ \lambda ) + \lambda ^2} = c_2 |s_k - (s- \lambda)^+|. \]

On the other hand, if $s_0 \geq s-\lambda$ then we have $s_0 \geq s_k \geq s-\lambda$ for all $k$, and setting $c_3 = \frac{\lambda^2}{s^2}<1$, \eqref{eqa} implies
\[ | s_{k+1} - (s-\lambda)^+| = \frac{\lambda ^2 |s_k - (s- \lambda)| } {s_k (s+ \lambda ) + \lambda ^2} \leq \frac{\lambda ^2 |s_k - (s- \lambda)| } {(s-\lambda) (s+ \lambda ) + \lambda ^2} = c_3 |s_k - (s- \lambda)^+|. \]
So in each case we have $| s_{k+1} - (s-\lambda)^+| \leq c|s_k - (s- \lambda)^+|$ for some $c \in [0,1)$.
\end{proof}
The above lemma establishes a linear convergence rate which is crucial when we consider the perturbed iteration below which will be critical to establishing convergence of the full iteration \eqref{Seq2} and \eqref{tSeq2}.  We first establish a general perturbation results for convergent sequences.

\begin{lemma}\label{perturb} Consider an iteration $x_{k+1} = f(x_k)$ with a fixed point $x^*$ such that for some $c\in [0,1)$ we have
\[ |f(x)- x^*| < c|x-x^*| \]
for all $x$.  Consider a sequence of perturbations $e_k$ such that for some $a\in[0,1)$ we have $|e_{k+1}|<a|e_k|$ then the perturbed sequence $w_{k+1} = f(w_k) + e_k$ converges to $x^*$ for any $w_0$.
\end{lemma}
\begin{proof} First, since $x^*=f(x^*)$ we have,
\begin{align*}
 |w_{k+1} - x^*| = |f(w_{k}) + e_{k} - x^* | \leq | f(w_{k}) - f(x^*) | + |e_{k}|   <  c |w_{k} - x^*| + |e_{k}|  
\end{align*}
and a simple induction shows that $|w_{k+1} - x^*| < \sum_{i=0}^k c^i |e_{k-i}|$.  Since $|e_{k+1}|<a|e_k|$ for all $k$, we have $|e_{k-i}| < a^{k-i}|e_0|$ and thus,
\[ |w_{k+1} - x^*| < \sum_{i=0}^k c^i |e_{k-i}| < |w_{k+1} - x^*| < |e_0| \sum_{i=0}^k c^i a^{k-i} = |e_0| \frac{a^{k+1}-c^{k+1}}{a-c} \to 0\]
since $c,a, \in [0,1)$, so $w_k \to x^*$.
\end{proof}
Note that when applying Lemma \ref{perturb} to the sequence $s_k$ of singular values, the required inequality on $f$ holds only on $(0,\infty)$, however the sequence of perturbations cannot cause the sequence to leave this set since the perturbed sequence is also a sequence of singular values.  

\subsection{Perturbation of Singular Values} 

We can now show that as $U_k \to U$, the singular values of \eqref{Seq2} and \eqref{tSeq2} are a perturbation of the iteration in Lemma \ref{softmax}.  This perturbed sequence will satisfy the assumptions of Lemma \ref{perturb} and thus will still converge to the softmax, $(s-\lambda)^+$.

Returning to \eqref{Seq2}, when $U_k \neq U$ by Theorem \ref{thm1} we can write $U_k = U + E_k$ where the perturbations $E_k$ decay linearly to zero, $||E_{k+1}|| < a ||E_k|| \rightarrow 0$ for some $a\in [0,1)$.  We can write \eqref{Seq2} as
\begin{align*}
U_{k}\tilde W_{k}\tilde S_{k}\tilde W_{k}\tilde V_{k}^\top &= VSU^\top U_k W_k S_k(S_k + \lambda I)^{-1} \\ &=  VSU^\top( U +E_k)  W_k S_k(S_k + \lambda I)^{-1}    \\ 
&=  VSU^\top U W_k S_k(S_k + \lambda I)^{-1} + VSU^\top E_k  W_k S_k(S_k + \lambda I)^{-1} 
\end{align*}
The first term above will be same as right-hand-side of \eqref{Seq3} and will simplify to give the right-hand-side of \eqref{Seq4}.  The second term has bound
\[ || VSU^\top E_k  W_k S_k(S_k + \lambda I)^{-1} || \leq  || S|| \, || E_k || \, ||S_k(S_k + \lambda I)^{-1} || <   || S|| \, || E_k ||  \] 
since $V,U^\top,W_k$ are orthogonal and $S_k (S_k + \lambda I)^{-1}$ is diagonal with diagonal entries less than $1$.  By Weyl's law for the stability of singular values under perturbation (see for example Theorem 1 of \cite{Stewart}) the singular values $\tilde s_k$ on the left-hand-side of \eqref{Seq3} are given by a perturbation $e_k$ of the right-hand-side \eqref{Seq4} bounded by $||S|| ||E_k||$.  The iteration for the true singular values becomes,
 \begin{align}
\tilde s_{k} &= s s_k(s_k + \lambda)^{-1}  +  e_k    \label{Seq5}  \\
s_{k+1}   &= s \tilde s_k(\tilde s_k + \lambda)^{-1} + \tilde e_k.  \label{tSeq5}
\end{align} 
where $|e_k| < ||S|| \, ||E_k||$ and by a similar we find a perturbation argument we have $|\tilde e_k| < ||S|| \, ||\tilde E_k||$.  Finally, the iteration \eqref{seq1} becomes,
\begin{align} s_{k+1}  &= \frac{s^2 s_k(s_k + \lambda)^{-1}  +  e_k}{s s_k(s_k + \lambda)^{-1}  +  e_k + \lambda} + \tilde e_k   = \frac{s^2 s_k  +  e_k(s_k+\lambda)}{ s_k(s+\lambda) +\lambda^2  +  e_k(s_k + \lambda)} + \tilde e_k  \nonumber \\
&= \frac{s^2 s_k }{ s_k(s+\lambda) +\lambda^2} + \hat e_k
\end{align} 
where
\begin{equation} \hat e_k = e_k\frac{(s_k + \lambda)(s_k(s+\lambda-s^2) +\lambda^2)}{(s_k(s+\lambda) +\lambda^2)(s_k(s+\lambda) +\lambda^2 + e_k(s_k + \lambda))} + \tilde e_k \end{equation}
Noting that $ s_k(s+\lambda-s^2) +\lambda^2 \leq s_k(s+\lambda) +\lambda^2$, we can estimate $\hat e_k$ as,
\[| \hat e_k | \leq | e_k | \left|\frac{ s_k + \lambda }{s_k(s+\lambda) +\lambda^2 + e_k(s_k + \lambda) } \right| +  |\tilde  e_k | \]
Since $e_k \to 0$, for $k$ sufficiently large we have $ - \lambda < e_k  < \lambda $.  We can bound the above denominator by, $ s_k(s+\lambda) +\lambda^2 + e_k(s_k + \lambda) > s_k(s+\lambda) +\lambda^2  - \lambda (s_k + \lambda) = s_k s $. Then,
\[| \hat e_k | \leq | e_k | \left|\frac{ s_k + \lambda  }{s_k s }\right| +  |\tilde  e_k |  \leq c |e_k| + |\tilde e_k| \]
since $s_k$ is bounded.  Since $e_k$ and $\tilde e_k$ have linear convergence, this implies that $\hat e_k$ has linear convergence as well.  Thus, by Lemma \ref{perturb} the true singular values, $s_k,\tilde s_k$ converge to the same limit as the unperturbed singular values, namely the soft max, $(s-\lambda)^+$.

\section{Effect of sign matrices on the cost functional}\label{signs}

We can now show that the matrices of right singular vectors $V_k,\tilde V_k$ from the SVDs in \eqref{svd1} and \eqref{svd2}, converge to diagonal sign matrices when $\lambda < S_{rr}$.  
\begin{theorem}\label{rightvecs} Let $X \in \mathbb{R}^{n\times m}$ have SVD $X=USV^\top$.  For $\lambda>0$ let $V_k, \tilde V_k$ be the sequence of matrices defined by \eqref{svd1} and \eqref{svd2}, then
\[ || \tilde V_k - I_{r\times p} ((S-\lambda I)^+ + \lambda I)  S^{-1} I_{p\times r} W_k ||_{\rm max} \to 0 \]
and when $W_k$ converges to a limit $W_*$ then $\tilde V_k \to I_{r\times p} ((S-\lambda I)^+ + \lambda I)  S^{-1} I_{p\times r} W_*$.  When $\lambda < S_{rr}$ we have $||\tilde V_k - W_k ||_{\rm max}\to 0$ and when $W_k\to W_*$ we have $V_k \to W_*$.  
\end{theorem}
\begin{proof}
Substituting \eqref{Beq} in \eqref{svd1} we have,
\[ \tilde U_k \tilde W_k \tilde D_k^2 \tilde W_k \tilde V_k^\top = X^\top U_k W_k D_k (D_k^2 +\lambda I)^{-1} D_k \]
where $X^\top U_k$ is $n\times r$ with $r \leq p \equiv \min\{m,n\}$.  In order to solve for $\tilde V_k^\top$ we multiply both sides by $U_k^\top X$ since $U_k^\top XX^\top U_k = U_k^\top US^2 U^\top U_k$ is invertible so that,
\[ U_k^\top X \tilde U_k \tilde W_k \tilde D_k^2 \tilde W_k = U_k^\top US^2 U^\top U_k W_k D_k (D_k^2 +\lambda I)^{-1} D_k \tilde V_k \]
and solving for $\tilde V_k$ yields,
\[ \tilde V_k =  D_k^{-2}(D_k^2+\lambda I) W_k(U_k^\top US^2 U^\top U_k)^{-1}U_k^\top USV^\top \tilde U_k \tilde D_k^2. \]
By Theorem \ref{thm1} we have $U_k^\top U \to I_{r\times p}$ and $V^\top \tilde U_k \to I_{p\times r}$ as $k\to \infty$ and as shown in Section \ref{valConvergence} we have $D_k \to I_{r\times p}DI_{p\times r} = I_{r\times p}(S-\lambda I)^+ I_{p\times r}$ and also $\tilde D_k \to I_{r\times p}DI_{p\times r}$.  Substituting these limits into the above equation gives the desired result.  Notice that when $\lambda < S_{rr}$ the maximum with zero has no effect and thus $((S-\lambda I)^+ + \lambda I)  S^{-1} = I$ so that $||\tilde V_k - W_k||_{\rm max} \to 0$.
\end{proof}
A similar argument shows that when $\lambda < S_{rr}$ we have $||V_k -\tilde W_k||_{\rm max} \to 0$ so that both $V_k,\tilde V_k$ are converging to diagonal sign matrices.  We can now characterize the convergence of Algorithm \ref{alg3}.

\begin{theorem}\label{finalthm} Let $X \in \mathbb{R}^{n\times m}$ have SVD $X=USV^\top$.  For $\lambda>0$, the iteration \eqref{Beq}-\eqref{svd2} converges whenever the sign matrices $W_k,\tilde W_k$ are chosen so that they converge to limits $W_k\to W_*$ and $\tilde W_k \to \tilde W_*$.  The cost \eqref{rrss} of the limiting matrices $A_*, B_*$ of the iteration is
\[ ||X-A_*B_*^\top||_F = ||S - (S-\lambda I)^+ S^2 ((S-\lambda I)^+ +\lambda I)^{-2} I_{p\times r} \tilde W_* W_*  I_{r\times p}||_F \]
and when $\lambda < S_{rr}$ it is
\[ ||X-A_*B_*^\top||_F = ||S - (S-\lambda I)^+ I_{p\times r} \tilde W_* W_*  I_{r\times p}||_F \]
 and only $\tilde W_* W_* = I$ will minimize the cost. 
\end{theorem}
\begin{proof}
If we make a convergent choice for the sign matrices $W_k \to W_*$ and $\tilde W_k \to \tilde W_*$ equation \eqref{Beq} defines a steady state,
\[ B_* = X^\top U_* W_* D_*(D_*^2+\lambda I)^{-1} = VS  D(D^2+\lambda I)^{-1}I_{p\times r} W_* \]
where $D_* = I_{r\times p}D I_{p\times r}$ as shown in Section \ref{valConvergence}.  Similarly \eqref{Aeq} defines a steady state,
\[ A_* = X \tilde U_* \tilde W_* D_*(D_*^2+\lambda I)^{-1} = US  D(D^2+\lambda I)^{-1}I_{p\times r} \tilde W_*. \]
Thus we find the low rank approximation of $X$ to be given by,
\[ A_{*}B_{*}^\top = US^2 D^2 (D^2+\lambda I)^{-2} I_{p\times r} \tilde W_* W_* I_{r\times p} V^\top \]
and when $\lambda < S_{rr}$ this reduces to 
\[ A_{*}B_{*}^\top = U (S-\lambda I)^+ I_{p\times r} \tilde W_* W_*  I_{r\times p} V^\top. \]
Notice that when $\tilde W_* W_* = I$ this is the optimal solution of \eqref{rrss} and \eqref{ssvd}.  In the general case, we find the first part of the cost functional is given by, 
\begin{align} ||X-A_*B_*^\top||_F &= ||USV^\top -U (S-\lambda I)^+ I_{p\times r} \tilde W_* W_*  I_{r\times p} V^\top||_F \nonumber \\
&= ||S - S^2 D^2 (D^2+\lambda I)^{-2} I_{p\times r} \tilde W_* W_*  I_{r\times p}||_F \nonumber
 \end{align}
 and when $\lambda < S_{rr}$ we have,
 \[ ||X-A_*B_*^\top||_F = ||S - (S-\lambda)^+ I_{p\times r} \tilde W_* W_*  I_{r\times p}||_F. \]
 Since $\tilde W_*$ and $W_*$ are diagonal sign matrices, so is $\tilde W_* W_*$ and any negative entries would change the subtraction to addition in the above cost functional, so the solution $A_* B_{*}^\top$ is optimal only when $\tilde W_* W_*  = I$.
\end{proof}
Finally, since $\tilde W_*$ and $W_*$ are both sign matrices, the way to insure $\tilde W_* W_* = I$ is to choose $W_* = \tilde W_*$.  In other words, we need to ensure that the choice of sign matrices in \eqref{svd1} and \eqref{svd2} are the same.  Algorithm \ref{alg3} does this by choosing the diagonal entries of $\tilde W_k$ to be the signs of the sums of the columns of $\tilde V_k$ and similarly for $W_k$ in terms of $V_k$.  Since Theorem \ref{rightvecs} show that $\tilde V_k, V_k$ are converging to diagonal matrices (independent of the choice of $\tilde W_k, W_k$) these choices of $\tilde W_k,W_k$ will insure that both $\tilde W_k \tilde V_k^\top$ and $W_k V_k$ converge to the identity matrix.  In fact, it does not matter which unique sign choice is made in the SVDs in \eqref{svd1} and \eqref{svd2} as long as the \emph{same} choice is made for both SVDs.  Effectively, the choice of sign matrices is how the right singular vectors of \eqref{svd1} and \eqref{svd2} contribute to the iteration in Algorithm \ref{alg3}, whereas they are not used at all in Algorithm \ref{alg2}.

\section{Conclusions and Future Work}

In this paper we introduced Algorithm \ref{alg3} as a new rank-restricted soft SVD method and we have proven convergence to the optimal solution of \eqref{rrss}.  We have shown that the standard method, Algorithm \ref{alg2}, can fail to converge or can converge to a non-optimal stationary point.  Moreover, we have derived the convergence rate of Algorithm \ref{alg3} based on the singular values of the matrix $X$ which shows how Algorithm \ref{alg3} can obtain much faster convergence than the naive alternating directions approach of Algorithm \ref{alg1}.  Since Algorithm \ref{alg3} is only one component of the matrix completion method introduced in \cite{Hastie}, an important future direction is analyzing the entire matrix completion algorithm.  Moreover, the choice of the rank restriction, $r$, and regularization parameter $\lambda$ are critical for obtaining the best matrix completion.  Investigating methods of selecting these parameters, possibly based on cross-validation, is another critical direction for future research.  Finally, while Algorithm \ref{alg3} is of significant interest due to its use in matrix completion problems \cite{Hastie,mazumder2010spectral,cai2010singular,candes2010matrix}, it could also be used as a partial SVD algorithm and comparison to state-of-the-art SVD methods \cite{SVD1,SVD2,SVD3} could yield future insights or improvements.

\bibliographystyle{siamplain}
\bibliography{biblio}
\end{document}